\documentclass[12pt]{amsart}
\usepackage{stmaryrd}
\usepackage{mathrsfs}
\usepackage{amsmath,amssymb}
\usepackage{amsfonts}
\usepackage{amsthm}
\usepackage{latexsym}
\usepackage{graphicx,color}
\def\p{\partial}

\def\ii{\sqrt{-1}}

\def\R{\mathbb{R}}
\def\C{\mathbb{C}}

\def\vv<#1>{\langle#1\rangle}
\def\ol{\overline}

\def\1{\mathbf{1}}

\def\Aut{{\rm Aut}}
\def\seq{\overset{=}{\sim}}

\def\XXint#1#2{\setbox0=\hbox{$#1{#2}{\int}$}{#2}\kern-.5\wd0 }

\def\XXint#1#2#3{{\setbox0=\hbox{$#1{#2#3}{\int}$}
     \vcenter{\hbox{$#2#3$}}\kern-.5\wd0}}

\def\vv<#1>{{\left\langle#1\right\rangle}}

\def\sph{\mathbb{S}}

\def\Deg{{\rm Deg}}

\def\Z{\mathbb{Z}}

\def\wt{\widetilde}

\def\Deg{{\rm Deg}}
\def\diag{{\rm diag}}
\def\abs{{\rm abs}}

\newtheorem{thm}{Theorem}[section]

\newtheorem{lem}{Lemma}[section]
\newtheorem{prop}{Proposition}[section]
\newtheorem{cor}{Corollary}[section]
\theoremstyle{definition}
\newtheorem{defn}{Definition}[section]
\theoremstyle{remark}

\newtheorem{rem}{Remark}[section]

\theoremstyle{question}

\theoremstyle{conjecture}

\numberwithin{equation}{section}

\begin{document}
\title{Weighted discrete tori and weighted trigonometric sums}

\author{Shuofeng Huang}
\address{Department of Mathematics, Shantou University, Shantou, Guangdong, 515063, China}
\email{24sfhuang@stu.edu.cn}
\author{Chengjie Yu$^1$}
\address{Department of Mathematics, Shantou University, Shantou, Guangdong, 515063, China}
\email{cjyu@stu.edu.cn}
\thanks{$^1$Research partially supported by GDNSF with contract no. 2025A1515011144 and 2026A1515012267.}
\renewcommand{\subjclassname}{%
  \textup{2010} Mathematics Subject Classification}
\date{}
\begin{abstract}
In this paper, we obtain a weighted trigonometric summation formula which is an extension of the trigonometric summation formula by Grigor'yan, Lin and Yau \cite{GLY}. 
\end{abstract}
\maketitle
\markboth{Huang \& Yu}{Weighted discrete tori and weighted trigonometric sums}
\section{Introduction}
In this paper, we obtain the following weighted trigonometric summation formula.
\begin{thm}\label{thm-main}Let $A$ be a nonsingular $d\times d$ integer matrix  and $w\in \R^d$. Then,
\begin{equation}\label{eq-weighted-sum}
\begin{split}
\sum_{\xi\in T_A^*}\vv<w,\cos(2\pi\xi)>^n=\frac{|\det A|}{2^n}\sum_{x\in A\Z^d}\sum_{y\in \mathscr{P}(\frac{n-|x|}{2},d)}\frac{n!}{(\abs(x)+y)!y!}w^{\abs(x)+2y}
\end{split}
\end{equation} 
for any $n\geq 1$. Here, $T_A^*=(A^{T})^{-1}\Z^d/\Z^d$ and for any $\xi\in T_A^*$, 
$$\cos(2\pi\xi):=\left(\cos(2\pi\xi_1),\cos(2\pi\xi_2),\cdots,\cos(2\pi\xi_d)\right)\in \R^d,$$
 $$\abs(x)=(|x_1|,|x_2|,\cdots,|x_d|)\mbox{ and }|x|=\sum_{i=1}^d|x_i|$$ 
for $x\in \Z^d$, and 
$$x!=x_1!x_2!\cdots x_d!\mbox{ and } w^x=w_1^{x_1} w_2^{x_2}\cdots w_d^{x_d}$$
for $x\in \Z_{\geq 0}^d$. Moreover, for any $k\in \Z_{\geq 0}$, 
$$\mathscr{P}(k,d)=\{x\in \Z^d_{\geq 0}\ |\ |x|=k\}.$$
When $k<0$ or $k\not\in \Z$, $\mathscr{P}(k,d)=\emptyset.$
\end{thm}
In Theorem \ref{thm-main}, when $w_1=w_2=\cdots=w_d=1$, the identity \eqref{eq-weighted-sum} becomes the trigonometric summation formula obtained by Grygor'yan-Lin-Yau \cite{GLY}. When $\det A=\pm 1$, $T_A^{*}=\{0\}$ and the identity \eqref{eq-weighted-sum} becomes the multinomial theorem:
$$(w_1+w_2+\cdots+w_d)^n=\sum_{m\in \mathscr{P}(n,d)}\frac{n!}{m!}w^m.$$
By applying the multinomial theorem to the LHS of \eqref{eq-weighted-sum}, and comparing the coefficients of the polynomials on $w_1,w_2,\cdots, w_d$ of the both sides of \eqref{eq-weighted-sum}, we have the following corollary.
\begin{cor}\label{cor-main}
Let $A$ be a nonsingular $d\times d$ integer matrix and $m\in \Z_{\geq 0}^d$. Then, 
\begin{equation}\label{eq-cor}
\sum_{\xi\in T_A^*}[\cos(2\pi\xi)]^m=\frac{|\det A|}{2^{|m|}}\sum_{\overset{x\in A\Z^d,y\in \Z_{\geq 0}^d}{\abs(x)+2y=m}}\frac{m!}{(\abs(x)+y)!y!}.
\end{equation}
\end{cor}

The proof of Theorem \ref{thm-main} is following the arguments in \cite{GLY} and \cite{CY}. It can be viewed as a discrete version of Poisson summation formula or a combinatorial trace formula. 

Let's recall the classical trace formula and Poisson summation formula (see \cite{Be}). Let $(M^n,g)$ be a closed Riemannian manifold, and 
$$0=\mu_1\leq \mu_2\leq \cdots\leq \mu_k\leq\cdots$$
be the Laplacian eigenvalues of $M$ counting multiplicities. Let $$\varphi_1,\varphi_2,\cdots, \varphi_k,\cdots$$
be an orthonormal system of real-valued eigenfunctions with 
$$\Delta \varphi_i=-\mu_i \varphi_i.$$
Then, the heat kernel $H(x,y,t)$ of $M$ can be written as 
$$H(x,y,t)=\sum_{i=1}^\infty e^{-\mu_it}\varphi_i(x)\varphi_i(y).$$
In particular, we have
\begin{equation}\label{eq-trace}
\sum_{i=1}^\infty e^{-\mu_i t}=\int_M H(x,x,t)dV_g(x).
\end{equation}
This is the classical trace formula.  When considering some special Riemannian manifolds where the spectrum data and heat kernel are explicitly known, we will get interesting summation formulas. For example, consider the unit circle $\sph^1=\R/2\pi \Z$. The Laplacian eigenvalues of $\sph^1$ are
\begin{equation}\label{eq-spec-S1}
\mu_1=0,\mu_2=\mu_3=1,\cdots, \mu_{2k}=\mu_{2k+1}=k^2,\cdots.
\end{equation}
The heat kernel of $\sph^1$ is 
$$H([x],[y],t)=\sum_{k\in \Z}p(x+2\pi k,y,t)=\frac{1}{\sqrt{4\pi t}}\sum_{k\in \Z}\exp\left(-\frac{|x+2\pi k-y|^2}{4t}\right)$$
where 
\begin{equation*}
p(x,y,t)=\frac{1}{\sqrt{4\pi t}}e^{-\frac{|x-y|^2}{4t}}
\end{equation*}
is the heat kernel of $\R$. So
$$H([x],[x],t)=\frac{1}{\sqrt{4\pi t}}\sum_{k\in \Z}\exp\left(-\frac{\pi^2k^2}{t}\right)$$
and 
\begin{equation}\label{eq-heat-trace-S1}
\int_{\sph^1}H([x],[x],t)ds=\sqrt{\frac{\pi}{t}}\sum_{k\in \Z}\exp\left(-\frac{\pi^2k^2}{t}\right).
\end{equation}
So, by substituting \eqref{eq-spec-S1} and \eqref{eq-heat-trace-S1} into \eqref{eq-trace}, we get the classical Poisson summation formula:
$$\sum_{k\in \Z}e^{-k^2t}=\sqrt{\frac{\pi}{t}}\sum_{k\in \Z}\exp\left(-\frac{\pi^2k^2}{t}\right).$$

The proof of Theorem \ref{thm-main} goes in the same way as the above. Let $(G,m,w)$ be a weighted finite graph, 
$$0=\mu_1\leq \mu_2\leq\cdots\leq\mu_{|V(G)|}$$
be its Laplacian eigenvalues, and $q_n(x,y)$ be the discrete-time heat kernel of $G$. Then, by a similar argument as the above, one has the discrete trace formula
\begin{equation}\label{eq-d-trace}
\sum_{i=1}^{|V(G)|}(1-\mu_i)^n=\sum_{x\in V(G)}q_n(x,x)m_x.
\end{equation}
For details, see Section 2. Then, by applying \eqref{eq-d-trace} to the discrete tori $\Z^d/A\Z^d$ equipped with a symmetric quotient weight where the both sides of \eqref{eq-d-trace} can be computed explicitly, one obtains \eqref{eq-weighted-sum}. For details, see Section 3. Some interesting analysis on discrete tori can be found in \cite{CJK,Fr,LWZ,Ve}.
\section{Preliminaries}
We first recall some preliminaries on analysis of graphs and derive the discrete trace formula. 

\begin{defn}
A triple $(G,m,w)$ is called a weighted graph if 
\begin{enumerate}
\item $G$ is a simple graph, 
\item $m:V(G)\to (0,+\infty)$ and 
\item $w:E(G)\to (0,+\infty)$.
\end{enumerate}
The function $m$ is called the vertex measure and the function $w$ is called the edge weight. 
\end{defn}

For convenience, we also view $w$ as a function on $V(G)\times V(G)$ which is also denoted as $w$ with 
$$w_{xy}=\left\{\begin{array}{ll}w(e)&e=\{x,y\}\\
0&\mbox{otherwise.}
\end{array}\right.$$
For two functions $f,g:V(G)\to \C$ with finite support, define their inner product as 
$$\vv<f,g>=\sum_{x\in V}f(x)\ol{g(x)}m_x.$$

Let $(G,m,w)$ be a weighted locally finite graph. Then, the Laplacian operator $\Delta$ on $G$ is defined as 
$$\Delta f(x)=\frac{1}{m_x}\sum_{y\in V(G)}(f(y)-f(x))w_{xy},\ \forall\ x\in V(G),$$
where $f\in \C^{V(G)}$. The Laplacian operator is formally self-adjoint:
$$\vv<\Delta f,g>=\vv<f,\Delta g>$$
for any $f,g:V(G)\to \C$ with finite support. 

A sequence of functions $q_n:V(G)\times V(G)\to \R$ with $n\geq 0$ such that 
\begin{equation}\label{eq-d-heat-kernel}
\left\{\begin{array}{ll}q_{n+1}(x,y)-q_{n}(x,y)=\Delta_x q_n(x,y)& n\geq 0\\
q_0(x,y)=\delta_y(x),
\end{array}\right.
\end{equation}
is called the discrete-time heat kernel of $G$. Here $$\delta_y(x)=\frac{1}{m_y}\1_y(x)$$
where
$$\1_y(x)=\left\{\begin{array}{ll}1&x=y\\
0&\mbox{otherwise.}
\end{array}\right.$$

When $(G,m,w)$ is a weighted finite graph, we have the following expression of the discrete-time heat kernel.
\begin{prop}
Let $(G,m,w)$ be a weighted finite graph on $N$ vertices, and $f_1,f_2,\cdots, f_N$ be a unitary system of Laplacian eigenfunctions with 
$$\Delta f_i=-\mu_i f_i.$$
Then, 
\begin{equation}\label{eq-heat-kernel-spec}
q_n(x,y)=\sum_{i=1}^N(1-\mu_i)^nf_i(x)\ol{f_i(y)},\ \forall\ n=0,1,2,\cdots.
\end{equation}
In particular, 
\begin{equation}\label{eq-d-trace-2}
\sum_{i=1}^N(1-\mu_i)^n=\sum_{x\in V(G)}q_n(x,x)m_x,\ \forall\ n=0,1,2,\cdots.
\end{equation}
\end{prop}
\begin{proof}
Let $Q_n(x,y)=\sum_{i=1}^N(1-\mu_i)^nf_i(x)\ol{f_i(y)}$. Then, 
\begin{equation*}
\begin{split}
Q_{n+1}(x,y)-Q_n(x,y)=&-\sum_{i=1}^N(1-\mu_i)^n\mu_i f_i(x)\ol{f_i(y)}\\
=&\sum_{i=1}^N(1-\mu_i)^n\Delta_xf_i(x)\ol{f_i(y)}\\
=&\Delta_xQ_n(x,y).
\end{split}
\end{equation*}
Moreover, for any $f:V(G)\to \C$, suppose that 
$$f=\sum_{i=1}^Nc_if_i.$$
Then, 
$$\vv<Q_0(\cdot,y),f>=\sum_{i=1}^N\vv<f_i,f>\ol {f_i(y)}=\sum_{i=1}^N\ol{c_i}\cdot \ol{\varphi_i(y)}=\ol{f(y)}=\vv<\delta_y,f>.$$
So, $Q_0(x,y)=\delta_y(x).$ Thus, the sequence of functions $Q_n(x,y)$ with $n\geq 0$ also satisfy \eqref{eq-d-heat-kernel} and hence 
$$q_n(x,y)=Q_n(x,y), \forall\ x,y\in V(G).$$
This gives us \eqref{eq-heat-kernel-spec}. 

Furthermore, by \eqref{eq-heat-kernel-spec}, we have
\begin{equation*}
\sum_{x\in V(G)}q_n(x,x)m_x=\sum_{x\in V(G)}\sum_{i=1}^N(1-\mu_i)^nf_i(x)\ol{f_i(x)}m_x=\sum_{i=1}^N(1-\mu_i)^n.
\end{equation*}
This completes the proof of the proposition.
\end{proof}
 
Next, we introduce a way to compute the discrete-time heat kernel $q_n(x,y)$ which is similar to that in \cite{Gr,Ba}. Here, we present the method for arbitrary weighted locally finite graphs.

By \eqref{eq-d-heat-kernel}, we have
\begin{equation}\label{eq-heat-kernel-re-1}
\begin{split}
q_{n+1}(x,y)=&q_n(x,y)+\Delta_xq_n(x,y)\\
=&q_n(x,y)+\frac{1}{m_x}\sum_{z\in V}(q_n(z,y)-q_n(x,y))w_{xz}\\
=&(1-\Deg(x))q_n(x,y)+\sum_{z\in V}\frac{w_{xz}}{m_x}q_n(z,y),
\end{split}
\end{equation}
where  $$\Deg(x)=\frac{1}{m_x}\sum_{y\in V}\frac{w_{xy}}{m_x}.$$ In particular, 
$$q_1(x,y)=(1-\Deg(x))q_0(x,y)+\frac{w_{xy}}{m_xm_y}=p_1(x,y)/m_y,$$
where 
\begin{equation}\label{eq-p-1}
p_1(x,y)=(1-\Deg(x))\1_y(x)+\frac{w_{xy}}{m_x}.
\end{equation}
Moreover, by \eqref{eq-heat-kernel-re-1},
\begin{equation}\label{eq-heat-kernel-re-2}
\begin{split}
q_{n+1}(x,y)=&(1-\Deg(x))q_n(x,y)+\sum_{z\in V}\frac{w_{xz}}{m_x}q_n(z,y)\\
=&\sum_{z\in V(G)}\left((1-\Deg(x))\1_x(z)+\frac{w_{xz}}{m_x}\right)q_n(z,y)\\
=&\sum_{z\in V(G)}\left((1-\Deg(x))\frac{\1_z(x)}{m_z}+\frac{w_{xz}}{m_xm_z}\right)q_n(z,y)m_z\\
=&\sum_{z\in V(G)}q_1(x,z)q_n(z,y)m_z.
\end{split}
\end{equation}
Let $P:\C^{V(G)}\to \C^{V(G)}$ be defined as 
$$(Pf)(x)=\sum_{y\in V(G)}p_1(x,y)f(y)$$
which is analogue of the Markov operator in \cite{Gr}. Suppose 
$$(P^nf)(x)=\sum_{y\in V(G)}p_n(x,y)f(y).$$
By this definition, it is natural to take $p_0(x,y)=\1_y(x)$. Then, by that $$P^{m+n}f=P^m(P^nf),$$ 
we have
\begin{equation}\label{eq-p-convolution}
p_{m+n}(x,y)=\sum_{z\in V(G)}p_m(x,z)p_n(z,y)
\end{equation}
for any nonnegative integers $m$ and $n$. Here $p_n(x,y)$ can be viewed as an analogue of the $n$-step transition probability from $x$ to $y$ in \cite{Gr, Ba}. 

Similar to that in \cite{Gr,Ba}, we have the following properties for $q_n$ and its relation to $p_n$.
\begin{lem}\label{lem-p-q}
Let the notations be the same as in the above. Then,
\begin{enumerate}
\item for any $x,y\in V(G)$ and nonnegative integer $n$,
$$q_n(x,y)=p_n(x,y)/m_y;$$
\item for any $x,y\in V(G)$ and nonnegative integers $m$ and $n$,
$$q_{m+n}(x,y)=\sum_{z\in V}q_{m}(x,z)q_n(z,y)m_z;$$
\item for any $x,y\in V(G)$ and nonnegative integers $n$,
$$q_n(x,y)=q_n(y,x).$$
\end{enumerate} 

\end{lem}
\begin{proof}
(1) We prove this by induction on $n$. It is true for $n=0,1$ by definition. Suppose the conclusion is true for $n=k$. When $n=k+1$, by \eqref{eq-heat-kernel-re-2} and \eqref{eq-p-convolution}, we have 
\begin{equation}
\begin{split}
q_{k+1}(x,y)=&\sum_{z\in V(G)}q_1(x,z)q_k(z,y)m_z\\
=&\sum_{z\in V(G)}p_1(x,z)p_k(z,y)/m_y\\
=&p_{k+1}(x,y)/m_y.
\end{split}
\end{equation}
So, we have the conclusion (1).\\

(2) By (1) and \eqref{eq-p-convolution}, we have
\begin{equation}
\begin{split}
q_{m+n}(x,y)=&p_{m+n}(x,y)/m_y\\
=&\sum_{z\in V(G)}(p_m(x,z)/m_z)(p_n(z,y)/m_y)m_z\\
=&\sum_{z\in V(G)}q_m(x,z)q_n(z,y)m_z.
\end{split}
\end{equation}

(3) We also prove it by induction. The conclusion is clearly true for $n=0,1$ by definition. Suppose the conclusion is true for $n=k$. When $n=k+1$,  by \eqref{eq-heat-kernel-re-2} and (2), we have
\begin{equation*}
\begin{split}
q_{k+1}(x,y)=&\sum_{z\in V(G)}q_1(x,z)q_k(z,y)m_z\\
=&\sum_{z\in V(G)}q_{k}(y,z)q_1(z,x)m_z\\
=&q_{k+1}(y,x).
\end{split}
\end{equation*}
\end{proof}
By (1) of Lemma \ref{lem-p-q}, to compute the discrete-time heat kernel $q_n$, we only need to compute the $n$-step transition probability $p_n$. Note that, by \eqref{eq-p-1}, $p_1(x,y)=0$ when $x\neq y$ and $x\not\sim y$, and when $\Deg(x)=1$ for any $x\in V(G)$, $p_1(x,y)=0$ for $x\not\sim y$. Combining these with \eqref{eq-p-convolution}, we have the following expression for $p_n(x,y)$.
\begin{cor} Let $(G,m,w)$ be a weighted locally finite graph. Then, for any $x,y\in V(G)$ and positive integer $n$,
\begin{equation}\label{eq-pn-g}
p_n(x,y)=\sum_{\omega\in \wt\Omega_n(x,y) }p_1(\omega_0,\omega_1)p_1(\omega_1,\omega_2)\cdots p_1(\omega_{n-1},\omega_n),
\end{equation}
where $\wt\Omega_n(x,y)$ is the collection of $n$-step generalized walks from $x$ to $y$. That is, 
\begin{equation*}
\wt\Omega_n(x,y)=\{\omega:x=\omega_0\seq \omega_1\seq \cdots\seq\omega_n=y|\omega_i\in V(G),\ \forall i=0,1,\cdots,n\}.
\end{equation*} 
Here $u\seq v$ means that $u\sim v$ or $u=v$. In particular, if $(G,m,w)$ is normalized weighted, that is $\Deg(v)=1$ for any $v\in V(G)$, then for any $x,y\in V(G)$ and positive integer $n$,
\begin{equation}\label{eq-pn-r}
p_n(x,y)=\sum_{\omega\in \Omega_n(x,y) }p_1(\omega_0,\omega_1)p_1(\omega_1,\omega_2)\cdots p_1(\omega_{n-1},\omega_n),
\end{equation}
where $\Omega_n(x,y)$ is the collection of $n$-step walks from $x$ to $y$. That is, 
\begin{equation*}
\Omega_n(x,y)=\{\omega:x=\omega_0\sim \omega_1\sim \cdots\sim\omega_n=y|\omega_i\in V(G),\ \forall i=0,1,\cdots,n\}.
\end{equation*} 
\end{cor}

Finally, we recall something about eigenfunctions and discrete-time heat kernel on quotient graphs which is an analogue of that in \cite{GLY}.
 
Let $(G,m,w)$ be a weighted locally finite graph. A bijection $\varphi:V(G)\to V(G)$ is called an automorphism of $(G,m,w)$ if 
$$m\circ \varphi=m \mbox{ and } w\circ(\varphi\times\varphi)=w\circ(\varphi\times\varphi).$$
It is clear that an automorphism of the weighted graph must be at least an graph automorphism. The collection of automorphisms for $(G,m,w)$ is denoted as $\Aut(G,m,w)$ which is called the automorphism group of $(G,m,w)$.

By definition, it is not hard to see that $\Deg$, $p_n$ and $q_n$ are invariant under automorphisms of $(G,m,w)$.
\begin{lem}\label{lem-p-q-invariant}
Let $(G,m,w)$ be a weighted locally finite graph and $\varphi\in \Aut(G,m,w)$. Then, for any $x,y\in V$,  $\Deg(\varphi(x))=\Deg(x)$,
$$p_n(\varphi(x),\varphi(y))=p_n(x,y)$$
and 
$$q_n(\varphi(x),\varphi(y))=q_n(x,y)$$
for any nonnegative integer $n$.
\end{lem}
\begin{proof}
First, 
\begin{equation*}
\begin{split}
\Deg(\varphi(x))=&\frac{1}{m(\varphi(x))}\sum_{y\in V(G)}w(\varphi(x),y)\\
=&\frac{1}{m(x)}\sum_{y\in V(G)}w(x,\varphi^{-1}(y))\\
=&\Deg(x).
\end{split}
\end{equation*}
Moreover, it is clear that $$p_0(\varphi(x),\varphi(y))=\1_{\varphi(y)}(\varphi(x))=\1_y(x)=p_0(x,y)$$
and
\begin{equation*}
\begin{split}
p_1(\varphi(x),\varphi(y))=&(1-\Deg(\varphi(x)))\1_{\varphi(y)}(\varphi(x))+\frac{w(\varphi(x),\varphi(y))}{m(\varphi(x))}\\
=&(1-\Deg(x))\1_{y}(x)+\frac{w(x,y)}{m(x)}\\
=&p_1(x,y).\\
\end{split}
\end{equation*}
So, $p_0$ and $p_1$ are invariant under $\varphi$. Suppose $p_n$ is invariant under $\varphi$ for $n=k$. Then, when $n=k+1$, by \eqref{eq-p-convolution}
\begin{equation}
\begin{split}
p_{k+1}(\varphi(x),\varphi(y))=&\sum_{z\in V(G)}p_{1}(\varphi(x),z)p_{k}(z,\varphi(y))\\
=&\sum_{z\in V(G)}p_{1}(x,\varphi ^{-1}(z))p_{k}(\varphi^{-1}(z),y)\\
=&p_{k+1}(x,y).
\end{split}
\end{equation} 
Finally, by (1) of Lemma \ref{lem-p-q}, 
$$q_n(\varphi(x),\varphi(y))=p_n(\varphi(x),\varphi(y))/m(\varphi(y))=p_n(x,y)/m(y)=q_n(x,y).$$
\end{proof}
Next, we recall the definition of weighted quotient graph in \cite{GLY}.
\begin{defn}\label{def-weighted-qg}
Let $(G,m,w)$ be a weighted locally finite graph and 
$\Gamma$ be a subgroup of $\Aut(G,m,\mu)$. Then, the weighted quotient graph 
$Q=G/\Gamma$ is defined as follows:
\begin{enumerate}
\item $V(Q):=V(G)/\Gamma$;
\item $E(Q):=\{\{[x],[y]\}\ |\ [x]\neq [y]\ \mbox{ and }E_G([x],[y])\neq\emptyset\}$;
\item for any $[x]\in V(Q)$,
$$m_Q([x]):=m_G(x);$$
\item for any $\{[x],[y]\}\in E(Q)$,
$$w_Q([x],[y]):=\sum_{g\in \Gamma}w_G(x,g_.y).$$
\end{enumerate}
\end{defn}
It is not hard to see that  the quotient weighted graph $(Q,m_Q,w_Q)$ is well-defined:
\begin{enumerate}
\item The vertex-measure $m_Q$ and edge-weight $w_Q$ are well-defined since $m_G$ and $w_G$ are invariant under $\Gamma$;
\item For any $[x],[y]\in V(Q)$,  $[x]\sim_Q[y]$ if and only if $w_Q([x],[y])> 0$.
\end{enumerate}

The following properties of weighted quotient graphs come directly from definition.
\begin{prop}\label{prop-qg}
Let $(G,m,w)$ be a weighted locally finite graph, $\Gamma$ be a subgroup of $\Aut(G,m,w)$ and $(Q,m_Q,w_Q)$ be the weighted quotient graph of $G$ over $\Gamma$. Then, 
\begin{enumerate}
\item for any $[x]\in V(Q)$, 
$$\Deg_Q([x])=\Deg_G(x);$$
\item for any $f\in \C^{V(Q)}$, 
$$\pi^*(\Delta_Qf)=\Delta_G(\pi^*f)$$
where $\pi:V(G)\to V(Q)$ is the quotient map;
\item if $G$ is a normalized weighted graph or $\Gamma$ acts on $G$ freely, then for any $[x],[y]\in V(Q)$ and positive integer $n$, 
$$p_n^Q([x],[y])=\sum_{g\in \Gamma}p^G_n(x,g_.y);$$
\item if $G$ is a normalized weighted graph or $\Gamma$ acts on $G$ freely, then for any $[x],[y]\in V(Q)$ and positive integer $n$,
$$q_n^Q([x],[y])=\sum_{g\in \Gamma}q_n^G(x,g_.y).$$
\end{enumerate}
\end{prop}
\begin{proof}
(1) For any $[x]\in V(Q)$, 
\begin{equation*}
\begin{split}
\Deg_Q([x])=&\frac{1}{m_Q([x])}\sum_{[y]\in V(Q)}w_Q([x],[y])\\
=&\frac{1}{m_G(x)}\sum_{[y]\in V(Q)}\sum_{g\in \Gamma}w_G(x,g_.y)\\
=&\Deg_G(x).
\end{split}
\end{equation*}

(2) For any $x\in V(G)$, 
\begin{equation}
\begin{split}
\pi^*(\Delta_Qf)(x)=&\Delta_Qf([x])\\
=&\frac{1}{m_Q([x])}\sum_{[y]\in V(Q)}(f([y])-f([x]))w_{Q}([x],[y])\\
=&\frac{1}{m_G(x)}\sum_{[y]\in V(Q)}(\pi^*f(y)-\pi^*f(x))\sum_{g\in \Gamma}w_G(x,g_.y)\\
=&\frac{1}{m_G(x)}\sum_{g\in \Gamma}\sum_{[y]\in V(Q)}(\pi^*f(g_.y)-\pi^*f(x))w_G(x,g_.y)\\
=&(\Delta_G\pi^*f)(x).
\end{split}
\end{equation}

(3) For any $[x],[y]\in V(Q)$, by \eqref{eq-p-1}, 
\begin{equation}
\begin{split}
p_1^Q([x],[y])=&(1-\Deg_Q([x]))\1_{[y]}([x])+\frac{w_Q([x],[y])}{m_Q([x])}\\
=&(1-\Deg_G(x))\1_{[y]}([x])+\frac{\sum_{g\in\Gamma}w_G(x,g_.y)}{m_G(x)}\\
\end{split}
\end{equation}
If $\Deg_G(x)=1$ for any $x\in V(G)$, then 
$$p_1^Q([x],[y])=\sum_{g\in \Gamma}\frac{w_G(x,g_.y)}{m_x}=\sum_{g\in \Gamma}p_1^G(x,g_.y).$$
If $\Gamma$ acts on $G$ freely, then 
$$\1_{[y]}([x])=\sum_{g\in \Gamma}\1_{g_.y}(x)$$
since if $[y]=[x]$, then there is a unique $g\in \Gamma$ such that $x=g_.y$. So, the same as before, we have
$$p_1^Q([x],[y])=\sum_{g\in \Gamma}p_1^G(x,g_.y).$$
Suppose that the conclusion is true for $n=k$. Then, when $n=k+1$, by \eqref{eq-p-convolution} and Lemma \ref{lem-p-q-invariant},
\begin{equation}
\begin{split}
p_{k+1}^Q([x],[y])=&\sum_{[z]\in V(Q)}p_1^Q([x],[z])p_{k}^Q([z],[y])\\
=&\sum_{[z]\in V(Q)}\left(\sum_{g\in \Gamma}p_1^G(x,g_.z)\right)\left(\sum_{h\in \Gamma}p_{k}^G(z,h_.y)\right)\\
=&\sum_{[z]\in V(Q)}\sum_{g,h\in \Gamma}p_1^G(x,g_.z)p_{k}^G(z,h_.y)\\
=&\sum_{[z]\in V(Q)}\sum_{g,h\in \Gamma}p_1^G(x,g_.z)p_{k}^G(g_.z,gh_.y)\\
=&\sum_{[z]\in V(Q)}\sum_{h,k\in \Gamma}p_1^G(x,kh^{-1}_.z)p_{k}^G(kh^{-1}_.z,k_.y)\\
=&\sum_{k\in \Gamma}\sum_{[z]\in V(Q)}\sum_{h\in \Gamma}p_1^G(x,kh^{-1}_.z)p_{k}^G(kh^{-1}_.z,k_.y)\\
=&\sum_{k\in \Gamma}p^G_{k+1}(x,k_.y).
\end{split}
\end{equation}
This completes the proof of (3).

(4) By (1) of Lemma \ref{lem-p-q} and (3), we have 
$$q_n^{Q}([x],[y])=p^Q_n([x],[y])/m_Q([y])=\sum_{g\in \Gamma}p^G_n(x,g_.y)/m_G(g_.y)=\sum_{g\in\Gamma}q_n^G(x,g_.y).$$
\end{proof}
\begin{rem}
Note that $$p_0^Q([x],[y])=1_{[y]}([x])=\sum_{g\in \Gamma}1_{g_.y}(x)=\sum_{g\in \Gamma}p_0^G(x,g_.y)$$
is only true when $\Gamma$ acts on $G$ freely.
\end{rem}
\section{Weighted tori and proof of Theorem \ref{thm-main}}
In this section, we compute the discrete-time heat kernel and the Laplacian spectrum of the weighted tori $T_A=\Z^d/A\Z^d$ (see the definition below), and apply \eqref{eq-d-trace-2} to prove Theorem \ref{thm-main}.

Let $\Z^d$ be the integer lattice graph and $w=(w_1,w_2,\cdots,w_d)\in \R^d_{>0}$.
Let 
$$w_{\{x,x+e_i\}}=w_i>0$$
for $i=1,2,\cdots, d$ and $x\in \Z^d$, and
$$m_x=\sigma:=2\sum_{i=1}^dw_i$$
for any $x\in \Z^d$. Then, $(\Z^d,m,w)$ is a normalized weighted graph.

For $z\in \Z^n$, denote the translation map $x\mapsto x+z$ as $\tau_z$. It is clear that 
$$\tau_z,\diag(\pm1,\cdots,\pm1)\in \Aut(\Z^d,m,w)$$
for any $z\in \Z^d$.
So, by Lemma \ref{lem-p-q-invariant}, 
\begin{equation}\label{eq-p-Zd-sym}
p_n^{\Z^d}(x,y)=p_n^{\Z^d}(0,y-x)=p_n^{\Z^d}(0,\abs(y-x))
\end{equation}
for any $x,y\in \Z^d$.

\begin{thm}
Let $(\Z^d,m,w)$ be the weighted integer lattice graph defined above. Then, for any $x,y\in \Z^d$ and positive integer $n$,
$$p_n^{\Z^d}(x,y)=\frac{1}{\sigma^n}\sum_{z\in \mathscr{P}(\frac{n-|y-x|}{2},d)}\frac{n!}{(\abs(y-x)+z)!z!}w^{\abs(y-x)+2z}$$
and
$$q_n^{\Z^d}(x,y)=\frac{1}{\sigma^{n+1}}\sum_{z\in \mathscr{P}(\frac{n-|y-x|}{2},d)}\frac{n!}{(\abs(y-x)+z)!z!}w^{\abs(y-x)+2z}.$$
\end{thm}
\begin{proof}
By \eqref{eq-p-Zd-sym} and (1) of Lemma \ref{lem-p-q}, we only need to prove the first formula for $x=0$ and $y\in \Z^d_{\geq 0}$. Let 
$$\omega:0=\omega_0\sim \omega_1\sim \cdots\sim \omega_n=y$$
be a walk joining $0$ to $y$ and  $u_i=\omega_{i}-\omega_{i-1}$. Then $u_i=\pm e_{k_i}$ for some $k_i=1,2,\cdots,d$, and 
$$u_1+u_2+\cdots+u_n=y.$$
Let $a_k$ be the number of $e_k$ in $u_1,u_2,\cdots, u_n$ and $z_k$ be the number of $-e_k$ in $u_1,u_2,\cdots,u_n$. Then,
$$\sum_{k=1}^da_k+\sum_{k=1}^d z_k=n$$
and 
$$a_k-z_k=y_k$$
for $k=1,2,\cdots,d$. It is then clear that 
$$z\in\mathscr{P}\left(\frac{n-|y|}{2},d\right).$$

Conversely, for any $z\in \mathscr{P}\left(\frac{n-|y|}{2},d\right)$, let $a=y+z$. Then, each element in 
$$\mathscr{U}(y,z):=\left\{u:=(u_1,u_2,\cdots,u_n)\bigg|\begin{array}{l}\mbox{$e_k$ and $-e_k$ appears $a_k$ and $z_k$ times in $u$}\\
\mbox{respectively, for $k=1,2,\cdots,d$.}
\end{array} \right\}$$
corresponds to a walk 
$$\omega:0=\omega_0\sim \omega_1\sim \cdots\sim \omega_n=y$$
with $u_i=\omega_{i}-\omega_{i-1}$. Therefore, by \eqref{eq-pn-r} and \eqref{eq-p-Zd-sym}, 
\begin{equation*}
\begin{split}
p_n^{\Z_n}(0,y)=&\sum_{\omega\in \Omega_n(0,y)}p_1^{\Z_n}(\omega_0,\omega_1)\p_1^{\Z^n}(\omega_1,\omega_2)\cdot p_1^{\Z^n}(\omega_{n-1},\omega_n)\\
=&\sum_{\omega\in \Omega_n(0,y)}p_1^{\Z_n}(0,\omega_1-\omega_0)p_1^{\Z^n}(0,\omega_2-\omega_1)\cdots p_1^{\Z^n}(0,\omega_n-\omega_{n-1})\\
=&\sum_{z\in \mathscr{P}(\frac{n-|y|}{2},d)}\sum_{u\in\mathscr{U}(y,z) }p_1^{\Z_n}(0,u_1)p_1^{\Z^n}(0,u_2)\cdots p_1^{\Z^n}(0,u_n)\\
=&\sum_{z\in \mathscr{P}(\frac{n-|y|}{2},d)}|\mathscr{U}(y,z)|\frac{w^{y+2z}}{\sigma^n}\\
=&\frac{1}{\sigma^n}\sum_{z\in \mathscr{P}(\frac{n-|y|}{2},d)}\frac{n!}{(y+z)!z!}w^{y+2z}.
\end{split}
\end{equation*}
This completes the proof of the theorem. 
\end{proof}

For any $d\times d$ nonsingular integer matrix $A$, $A\Z^d$ is a lattice  in $\Z^d$. Note that the weighted graph $(\Z^d,m,w)$ is invariant under translations. It induces a weighted quotient graph on $T_A=\Z^d/A\Z^d$ as in Definition \ref{def-weighted-qg}. We call $T_A$ the weighted discrete tori induced by $A$. By (4) of Proposition \ref{prop-qg}, we have the following conclusion.
\begin{cor}\label{cor-q-TA}
Let $A$ be a $d\times d$ nonsingular integer matrix and $T_A=\Z^d/A\Z^d$ be the weighted discrete tori induced by $A$ defined above. Then, for any $[x],[y]\in T_A$ and $n\geq 1$,
\begin{equation}
q_n^{T_A}([x],[y])=\frac{1}{\sigma^{n+1}}\sum_{g\in A\Z^d}\sum_{z\in \mathscr{P}(\frac{n-|y+g-x|}{2},d)}\frac{n!}{(\abs(y+g-x)+z)!z!}w^{\abs(y+g-x)+2z}.
\end{equation}
In particular, 
\begin{equation}
q_n^{T_A}([x],[x])=\frac{1}{\sigma^{n+1}}\sum_{g\in A\Z^d}\sum_{z\in \mathscr{P}(\frac{n-|g|}{2},d)}\frac{n!}{(\abs(g)+z)!z!}w^{\abs(g)+2z}
\end{equation}
for any $[x]\in T_A$ and $n\geq 1$.
\end{cor}
Next, we imitate the method in \cite{GLY} to compute the Laplacian spectrum of $T_A$. 
\begin{prop}\label{prop-spec-TA}
Let $A$ be a $d\times d$ nonsingular integer matrix and $T_A$ be the weighted discrete torus induced by $A$. Then,  
\begin{enumerate}
\item for any $\xi\in T^*_{A}$, $f_\xi(x)=e^{2\pi\ii\vv<\xi,x>}$ is a Laplacian eigenfunction of $T_A$ with eigenvalue $\left(1-\frac{2}{\sigma}\vv<w,\cos(2\pi\xi)>\right)$;
\item the family of functions $f_\xi$ with $\xi\in T_A^*$ are orthogonal of each other;
\item the Laplacian spectrum of $T_A$ consists of $\left(1-\frac{2}{\sigma}\vv<w,\cos(2\pi\xi)>\right)$ with $\xi\in T_A^*$.
\end{enumerate}
\end{prop}
\begin{proof}
(1) The same as the argument in the proof of \cite[Lemma 3.1]{GLY}, we know that $f_\xi$ is invariant under $A\Z^d$, So $f_\xi$ can be viewed as a function on $T_A$. Moreover, 
note that 
\begin{equation*}
\begin{split}
\Delta_{\Z^d} f_\xi(x)
=&\frac{1}{m(x)}\sum_{y\in \Z^d}(f_\xi(y)-f_\xi(x))w_{xy}\\
=&\frac{1}{\sigma}\sum_{i=1}^d(f_\xi(x+e_i)+f_\xi(x-e_i)-2f_\xi(x))w_i\\
=&\frac{1}{\sigma}\sum_{i=1}^d\left( e^{2\pi\ii \vv<\xi,(x+e_i)>}+e^{2\pi\ii \vv<\xi,(x-e_i)>}-2e^{2\pi\ii \vv<\xi,x>}\right)w_i\\
=&-\left(1-\frac{2}{\sigma}\sum_{i=1}^dw_i\cos(2\pi\xi_i)\right)f_\xi(x).
\end{split}
\end{equation*}
Then, by (2) of Proposition \ref{prop-qg}, we complete the proof of (1).

(2) The conclusion is proved  in \cite[Lemma 3.1]{GLY}.

(3) Note that $|T_A|=|T_{A}^*|=|\det A|$. Combining this with (1) and (2), we get the conclusion.
\end{proof}
We are now ready to prove Theorem \ref{thm-main} by using \eqref{eq-d-trace-2}.
\begin{proof}[Proof of Theorem \ref{thm-main}] Applying \eqref{eq-d-trace-2} to the weighted discrete tori $T_A$, and using Corollary \ref{cor-q-TA} and Proposition \ref{prop-spec-TA}, we have 
\begin{equation*}
\begin{split}
&\sum_{\xi\in T_A^*}\left(\frac{2}{\sigma}\vv<w,\cos(2\pi\xi)>\right)^n\\
=&\sum_{[x]\in T_A}q_n^{T_A}([x],[x])m_{T_A}([x])\\
=&|\det A|\sigma \times\frac{1}{\sigma^{n+1}}\sum_{g\in A\Z^d}\sum_{z\in \mathscr{P}(\frac{n-|g|}{2},d)}\frac{n!}{(\abs(g)+z)!z!}w^{\abs(g)+2z} \\
=&\frac{|\det A|}{\sigma^{n}}\sum_{g\in A\Z^d}\sum_{z\in \mathscr{P}(\frac{n-|g|}{2},d)}\frac{n!}{(\abs(g)+z)!z!}w^{\abs(g)+2z} \\
\end{split}
\end{equation*}
for $n\geq 1$. This gives us \eqref{eq-weighted-sum} for $w_1,w_2,\cdots, w_d>0$. Note that the both sides of \eqref{eq-weighted-sum} are polynomials on $w_1,w_2,\cdots,w_d$. So, the polynomials of the both sides must be the same and this gives us \eqref{eq-weighted-sum} for any $w\in\R^d$. This completes the proof of Theorem \ref{thm-main}.
\end{proof}

\end{document}